\documentclass[10pt]{amsart}
\usepackage[spanish, english]{babel}
\usepackage{amsmath,amssymb,latexsym,amsthm}
\usepackage[normalem]{ulem}
\usepackage{epsfig}
\usepackage{graphicx}
\usepackage{wrapfig}
\usepackage{color}

\newtheorem{thm}{Theorem}[section]
\newtheorem{lem}[thm]{Lemma}

\theoremstyle{definition}

\newtheorem{teo}{Theorem}

\newtheorem{prp}{Proposition}

\newtheorem{cor}{Corollary}

\newtheorem{con}{Conjecture}

\newtheorem{dfn}{Definition}

\newtheorem{rem}{Remark}

\title[On the sectional curvature along central configurations]
{On the sectional curvature along central configurations} 

\author[Connor Jackman]{Connor Jackman}
\address[Jackman]{Mathematics Department, University of California,
4111 McHenry
Santa Cruz, CA 95064, USA}
\email{cfjackma@ucsc.edu}

\author[Josu\'e Mel\'endez]{Josu\'e Mel\'endez}
\address[Mel\'endez]{Departamento de Matem\'aticas, UAM-Iztapalapa, 09340, M\'exico}
\email{jms@xanum.uam.mx}

\begin{document}
\date{\today}
\maketitle
\begin{abstract} 
In this paper we characterize planar central configurations in terms of a sectional curvature value of the Jacobi-Maupertuis metric. This characterization works for the $N$-body problem with general masses and any $1/r^{\alpha}$ potential with $\alpha> 0$. We also obtain dynamical consequences of these curvature values for relative equilibrium solutions. These curvature methods work well for strong forces ($\alpha \ge 2$).
\end{abstract}

\section{Introduction}

Since Euler and Lagrange, central configurations form a main theme in studies of the $N$-body problem.  Such configurations are characterized by the property that upon dropping the bodies from rest, they homothetically shrink to a total collision. Those of the 3-body problem have long been well known. They are exactly equilateral triangles due to Lagrange and certain collinear configurations found by Euler.

Every central configuration leads to \textit{homographic} solutions of the $N$-body problem: ones whose configuration evolves only by rotation or scaling. Such solutions are the only explicit solutions known for the $N>2$ body problems. See the nice references \cite{ Chenc, MoeckelCC, Mblowup} for more on central configurations.


By the {\em planar $N$-body problem with a $1/r^\alpha$ potential}, we mean the movement of $N$ point masses $q_1,...,q_N\in \mathbb{C}$, under the equations of motion \begin{equation}\label{eq:mov}
    m_k\ddot q_k = \frac{\partial U}{\partial q_k}
\end{equation} where $$U := \frac{1}{\alpha}\sum_{i<j} \frac{m_im_j}{|q_i-q_j|^\alpha}$$ is a $1/r^\alpha$ potential\footnote{For $\alpha=0$ one normally takes $U=-\sum_{i<j} m_im_j \log |q_i-q_j|$. Here we will consider $\alpha > 0$.} and $m_k>0$ are the masses. A {\em central configuration} is a configuration $q=(q_1,...,q_N)$ such that $\nabla U(q) = \lambda q$ for some $\lambda\in \mathbb{R}$.

Here we will study these central configurations using the {\em Jacobi-Maupertuis principle}. This principle reparametrizes solutions of a natural mechanical system at a fixed energy level as geodesics of a certain metric (eq. \eqref{JM} below), which we call the JM-metric for short.

It is well known that the sectional curvature values of a Riemannian metric at a point determine the local behavior of nearby geodesics, governing  the convergence or divergence of neighboring geodesics passing through this point. Namely, negative values imply nearby geodesics diverge more than those in a Euclidean plane, while positive values imply they diverge less (see Figure \ref{fig2}). These comparisons are made by examining growths of Jacobi Fields, which in turn are used to express the linearization of the geodesic flow, see for instance \cite{Arnold, Lee, Paternain}.  In this paper we examine some sectional curvatures of the JM-metric at central configurations.

\begin{figure}[h]
\begin{center}
\includegraphics[width=0.65\textwidth]{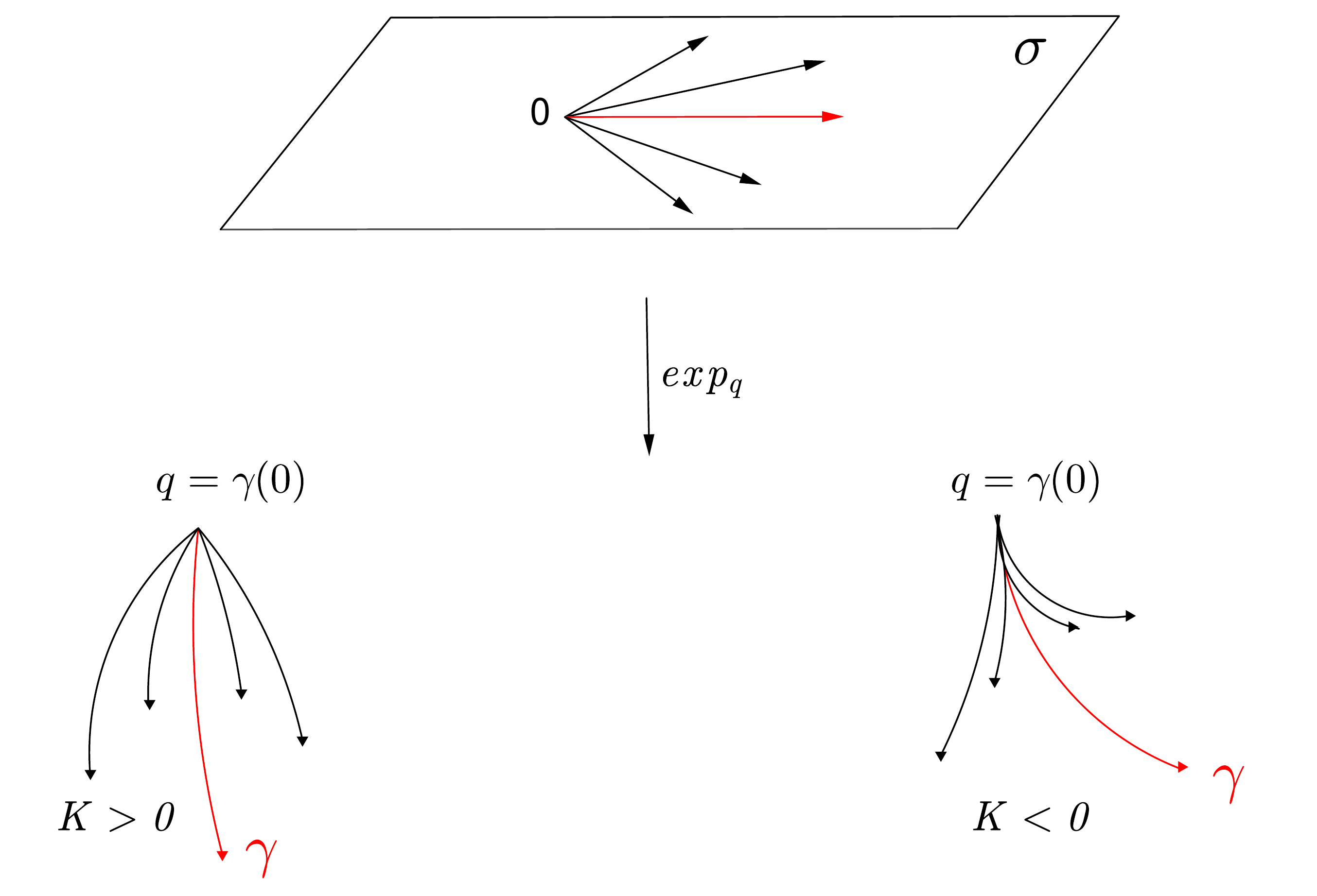}
\end{center}
\caption{This figure depicts the local behavior of geodesics on the 2-dimensional submanifold $\Sigma$ consisting of geodesic segments tangent to a plane $\sigma\subset T_qM$. When the sectional curvature $K_q(\sigma)$ is positive geodesics spread apart less than the corresponding rays of $\sigma$ (on the left), while they spread apart more when $K_q(\sigma)<0$ (on the right).}
\label{fig2}
\end{figure}

The idea of computing curvatures of the JM-metric to obtain dynamical consequences has been explored in other works. For three equal masses subject to a strong force (with $\alpha=2$), R. Montgomery showed in \cite{M2005} that, upon reducing by symmetries, the JM-curvatures are negative. This negative curvature allows a symbolic dynamics description of the orbits. Following  Montgomery's work,
the authors have found in \cite{JM} similar negatively curved circumstances for the collinear and parallelogram subproblems of the 4-body strong force problem.

With regard to central configurations, in his thesis \cite{Ong}, Ong Chong Pin examined curvatures of the JM-metric for central force problems, and over the Lagrange relative equilibrium solution of the classic ($\alpha=1$) 3-body problem. Later, M. Barbosu and B. Elmabsout in \cite{Barbosu} also computed some sectional curvatures along this Lagrange solution with equal masses, observing that certain energy values lead to negative sectional curvatures over this solution and noting that this negative curvature leads to `instability' of these solutions. More precisely, negative sectional curvatures over a suitable set of planes along the orbit leads to exponential growth of certain Jacobi fields in forwards or backwards time. Consequently the linearized Poincar\'e return map of the geodesic flow may have eigenvalues off of the unit circle. This instability is called \textit{spectral instability}:

\begin{dfn}
A periodic orbit is spectrally unstable if its linearized first return map has an eigenvalue $\lambda$ with $|\lambda|\ne 1$.
\end{dfn}

There is also the weaker notion of \textit{linear instability}:

\begin{dfn}
A periodic orbit is linearly unstable if its linearized first return map is not diagonalizable or it is spectrally unstable.
\end{dfn}


Since reparametrizing orbits has no effect on the Poincar\'e first return map, computations using the geodesic flow of the JM-metric will lead to the same eigenvalues as those associated to the flow of the original equations of motion on a fixed energy level. The reference \cite{Kob} allows a computation of the sectional curvatures through complex planes (Proposition \ref{prop3} below), from which we obtain the following  Barbosu-Elmabsout inspired result:

\begin{teo}
\label{inst}
 Consider the planar $N$-body problem under a strong force $1/r^\alpha$ potential with $\alpha\ge2$ and with masses $m_k>0$.
 
When $\alpha > 2$, all relative equilibria are spectrally unstable.

When $\alpha = 2$, all relative equilibria are linearly unstable (after reductions).
 
\end{teo}

\begin{rem}
\label{rem1}

In general, the homographic solutions at a positive energy level lie in a non-compact negatively curved totally geodesic surface, $\mathbb{C} q\backslash 0$, and the orbits are the geodesics on this surface under the JM-metric (see figure \ref{unbd}). Bounded motions with positive energy are only possible when we have $\alpha > 2$, while for $\alpha<2$ bounded motion occurs only for negative energies. The loss of this negatively curved surface containing periodic homographic motions, makes our curvature method harder to apply to $\alpha <2$ potentials.

The rotation and translation symmetries lead to 6 eigenvalues equal to 1 of the Poincar\'e map, whose eigenspaces are associated to the variations tangent to $\text{span}\{ iq, (1,...,1), (i,...,i)\}$. Center of mass drift gives linear instability in the $\text{span}\{(1,...,1), (i,...,i)\}$ directions. When $\alpha =2$ we find linear instability in the $\text{span}\{ iq, (1,...,1), (i,...,i)\}^{\perp}$ directions (those remaining after reductions).

\end{rem}

\begin{figure}[h]
\begin{center}
\includegraphics[width=0.65\textwidth]{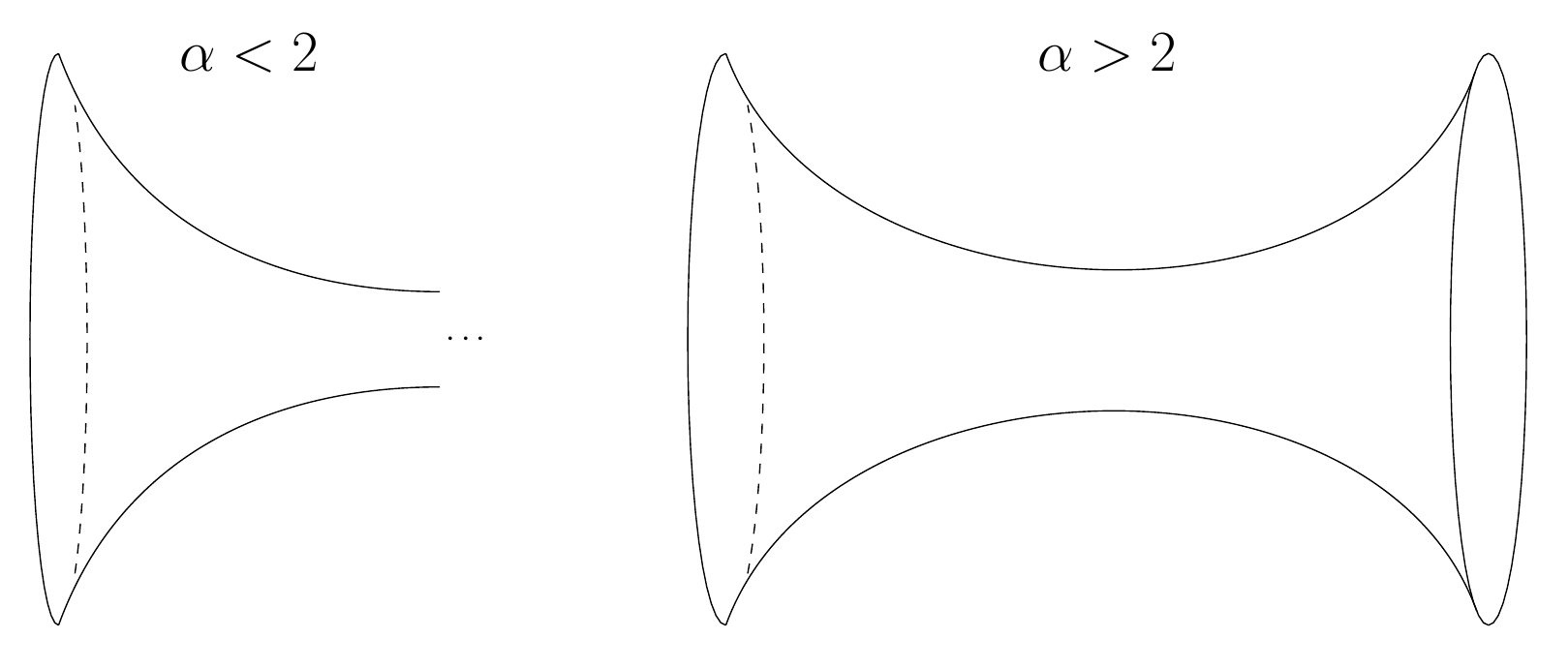}
\end{center}
\caption{The surfaces $\mathbb{C} q\backslash 0$ when $h>0$ are, under the JM-metric, totally geodesic surfaces of negative curvature. When $\alpha>2$, the `waist' is the relative equilibrium solution. The boundaries correspond to collision or escape to infinity.}
\label{unbd}
\end{figure}

\begin{con}
\label{con}

Consider the planar $N$-body problem under a $1/r^\alpha$ potential with masses $m_k>0$, and $0<\alpha<2$. Let $q$ be a central configuration.

Then at least one of the two classes of periodic homographic motions, $z(t)q$, are spectrally unstable:

(i) those near the circular relative equilibrium solution,

(ii) those near the homothetic `total collapse' solution.

\end{con}

\begin{rem}
\label{rem} Investigating stability properties of homographic solutions has been well studied, (see e.g. \cite{problems}: problem 15). In particular the Lagrange configuration when $\alpha =1$ by perturbation and numerical methods (see \cite{MS, Roberts} and references therein).  More recently, Hu and Sun \cite{Hu} have applied Maslov-type index theory to study the Lagrange motions (see also \cite{Bar, Aniso} for other recent applications of Maslov-type index theory).

It can be shown (see remark \ref{conj}) that for any given central configuration $q$, there is either a family of planes with negative sectional curvature tangent to the relative equilibrium solution through $q$ or a family of planes with mostly negative sectional curvatures tangent to the homothetic total collapse solution through $q$. This motivated the conjecture, since we expect the negative curvatures over these planes to give growth of Jacobi fields and instability. However, it remains to be seen whether the family of planes lies tangent to a family of geodesics i.e. whether there are Jacobi fields staying close to the family of planes for one period of the motion.

For comparison, figure \ref{huinst} (established by Martinez et. al \cite{MS}) shows detailed stability properties of the Lagrange solution for the classical $\alpha =1$ force law, and fits with our conjecture.

Conjecture \ref{con} may also be stated in terms of the scale invariant Dziobek constant: $D_{\alpha}:=hC^{2\alpha/(2-\alpha)}$, where $h$ is the energy and $C$ the angular momentum. Let $m$ be the value of $D_{\alpha}$ attained at the relative equilibrium solution over $q$. Then the conditions of $(i), (ii)$ translate to: the existence of an $\epsilon>0$ such that: $(i)~~D_{\alpha}<0$ and $|D_{\alpha}-m|<\epsilon$;  $(ii)~~D_{\alpha}\in (-\epsilon,0]$.

\end{rem}


\begin{figure}[h]
\begin{center}
\includegraphics[width=0.3\textwidth]{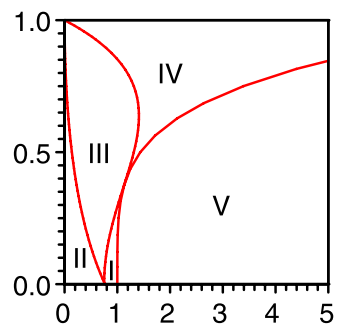}
\end{center}
\caption{This figure (6.1 from \cite{Hu}) shows stability properties of Lagrange solutions for $\alpha=1$ and different values of eccentricities (vertical axis) and masses (the horizontal axis is the mass parameter $\beta=27\frac{m_1m_2+m_1m_3+m_2m_3}{(m_1+m_2+m_3)^2}$).  Eccentricity zero is the relative equlibrium and eccentricity one is the total collapse solution. Spectral instability occurs in the regions III, IV, V.}
\label{huinst}
\end{figure}

While computing these sectional curvatures we also find the following characterization interesting:

\begin{prp}
\label{prop3}
Consider the JM-metric for a $1/r^\alpha$ potential with $\alpha> 0$ at energy level $h$ and let $K_q(u,v)$ denote the sectional curvature of the JM-metric at $q$ through the plane spanned by $u,v$.

The configuration $q$ is a central configuration if and only if $K_q(q, iq)=-\frac{h\alpha^2 U(q)}{2(h+U(q))^3\| q\|^2}.$

\end{prp}

\begin{rem}
\label{rem2}
The phase space for solutions to eq. \eqref{eq:motion} is $T(\mathbb{C}^N\backslash\Delta)$. If $q$ is a central configuration, then $T(\mathbb{C}q\backslash 0)$ is an invariant subspace. The dynamics restricted to this subspace is that of a $1/r^\alpha$ central force problem. The JM-metric associated to such a problem is $(h+r^{-\alpha})dzd\overline z$, and has Gaussian curvature $-\frac{h\alpha^2 r^{2\alpha-2}}{2(hr^\alpha+1)^3}$. This makes one direction of Proposition \ref{prop3} natural, while conversely we find it interesting that this curvature value in fact determines the central configurations.
\end{rem}

We also consider relating sectional curvatures to Saari's conjectures (see \cite{Fuji15}). In center of mass zero coordinates and with a Newtonian potential ($\alpha=1$), \textit{Saari's original conjecture} poses that the constancy of moment of inertia, $I:=\sum m_k |q_k|^2$, over a solution is equivalent to the solution being a relative equilibrium solution (see figure \ref{fig1}).

\begin{figure}[h]
\begin{center}
\includegraphics[width=0.35\textwidth]{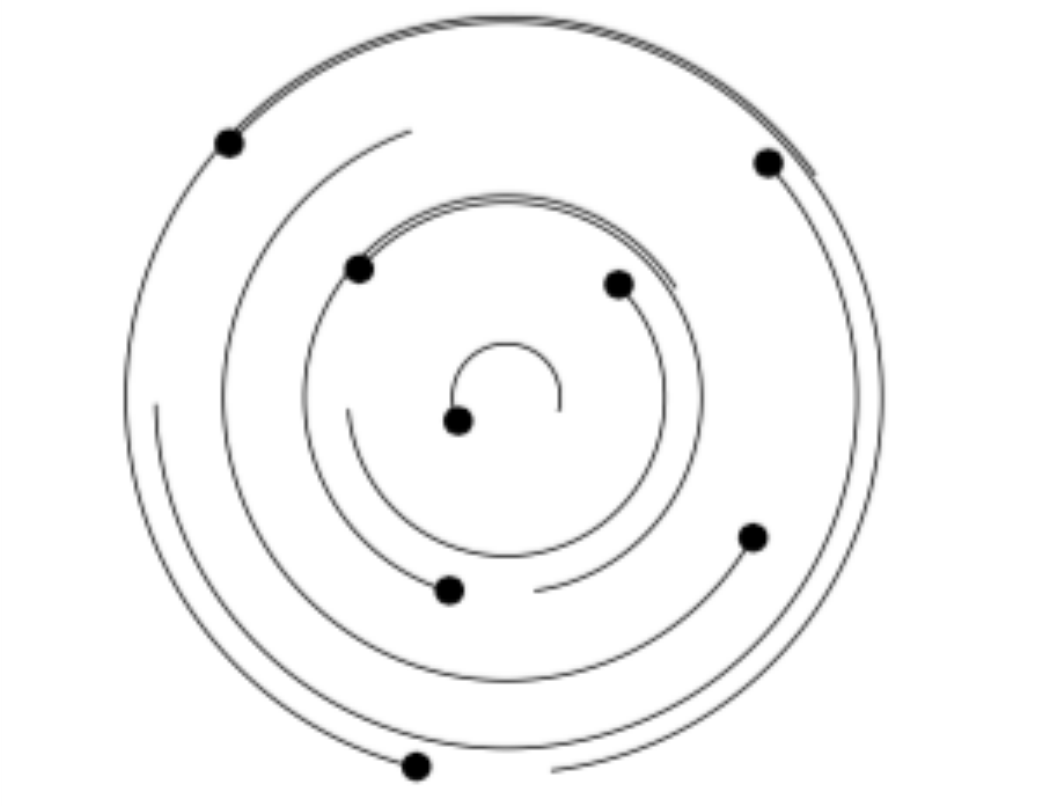}
\end{center}
\caption{This figure from \cite{MoeckelCC} depicts a relative equilibria solution.}
\label{fig1}
\end{figure}

This conjecture was settled affirmatively by Moeckel \cite{MoS}
for three bodies, and \textit{Saari's generalized conjecture} is the same statement extended to general $1/r^\alpha$-potentials. Saari then posed his \textit{homographic conjecture}: that constancy of $$\mu = I^{\alpha/2}U$$ over a solution is equivalent to the solution being homographic. Recall that a solution is called homographic if the configuration, $q(t)$, remains similar to the original configuration, $q(0)$; two configurations being similar if it is possible to pass from one to the other by a scaling transformation and a rotation.

The Lagrange-Jacobi identity \begin{equation}\label{eq:LJ}
 \ddot I = 4H + (4-2\alpha) U
\end{equation}
 shows that the `strong force' potential $\alpha=2$ is exceptional, in particular Saari's general conjecture is false for the strong force (see \cite{GR}) and for $\alpha\ne 2$ we have $\ddot I=const. \iff U=Const.$ Hence by restricting attention to relative equilibria, Saari's generalized conjecture reads (for any $\alpha$): \begin{center}\textit{The constancy of $U$ and $I$ over a solution is equivalent to the solution being a relative equilibrium.}\end{center}

Roughly speaking we seek to detect relative equilibria through sectional curvatures. To interpret Saari's conjectures on relative equilibria geometrically we are motivated by the following plan. First we seek some distribution of planes $\sigma$ and values $C(q)$ such that $K_{q(t)}(\sigma_{q(t)})=C(q(t))$ over a solution to eq. \eqref{eq:mov} is equivalent to the constancy of $U(q(t))$ and $I(q(t))$ over the solution. Next one would seek a result similar to proposition \ref{prop3} relating these curvature values to being a central configuration. Appealing as such a plan sounds, at the moment, it appears difficult and here we only explore the first step of this program.

\begin{prp}
\label{prop1}Let $\alpha>0.$
Let $q(t)$ be a solution of the $N$-body problem with an attractive $1/r^\alpha$ potential and $\dot{q}(t)\neq 0$. Let $ \textbf 1\in \mathbb{C}^N$ be the constant vector having the complex numbers $1+i$ in each coordinate.
If $U(q(t))$ is constant, then the sectional curvature $\displaystyle{K_{q(t)}(\dot{q}(t), \textbf 1)}$ of the  Jacobi-Maupertuis metric   is zero.
\end{prp}

\begin{prp}
\label{prop2}Let $\alpha=2$.
Let $q(t)$ be a solution of the planar $N$-body problem  with moment of inertia constant and  $\displaystyle{K_{q(t)}(\dot{q}(t), \textbf 1)}=0$,  then the potential  energy is constant.
\end{prp}

It is worth pointing out that Proposition \ref{prop2} remains true even if one replaces the condition $I(t)= $ constant by the condition that $q(t)$ is a bounded solution because, by the Lagrange-Jacobi identity with $\alpha = 2$, every bounded solution must have zero energy and constant moment of inertia $I$.

As for the condition on $\alpha$ in Proposition \ref{prop2}, it is well known that for $\alpha\neq 2$, and again using the Lagrange-Jacobi identity that the constancy of $I$ over a solution $q(t)$ implies that the potential energy $U$ is constant over the solution.

Finally, as a direct application of Proposition \ref{prop2} and Theorem 1 in \cite{Fuji15} (which we reference here as Theorem \ref{fuji} below), we obtain the following consequence.

\begin{cor}
\label{coro1}Let $\alpha=2$.
Let $q(t)$ be a  solution of the planar $3$-body problem  with  moment of inertia constant  and $\displaystyle{K_{q(t)}(\dot{q}(t), \textbf 1)}=0$, then $q(t)$ is a relative equilibria.
\end{cor}

\section{Notations}

Consider $N$ point particles of mass $m_k>0$ and positions $q_k\in \mathbb{C}$. The configuration of the system is described by the vector
\[
q=(q_1,\dots  q_N) \in \mathbb{C}^{N}\setminus\triangle
\]
where
$$\triangle =\{q =(q_1, \dots , q_N)\in \mathbb{C}^{N} : q_k=q_l, k\neq l\}$$
consists of all the collisions.

The mass-weighted Hermitian inner product
\begin{equation*}
 \left\langle u,v  \right\rangle_{\mathbb{C}}:=\sum_{k=1}^{N} m_k u_k \bar{v}_k.
\end{equation*}
allows us to write many formulae of celestial mechanics. We call its real and imaginary parts
\begin{equation*}
 \left\langle u,v \right\rangle=\text{Re} \left\langle u,v \right\rangle_{\mathbb{C}} \quad \qquad \omega(u,v)=\text{Im} \left\langle u,v \right\rangle_{\mathbb{C}}
\end{equation*}
the \textit{mass metric} and \textit{mass symplectic structure} on $\mathbb{C}^N$, which are $\mathbb{R}$-bilinear when we restrict scalar multiplication on $\mathbb{C}^N$ to real scalars.

The equations of motion (eq. \ref{eq:mov}) may then be written as
\begin{equation}\label{eq:motion}
\ddot{q}= \nabla U(q), \qquad \langle \nabla U, \cdot\rangle = dU(\cdot)
\end{equation}
where $\nabla$ is the gradient for the mass metric and $U(q)$ is a $1/r^\alpha$ potential. These equations of motion can as well be thought of as the Hamiltonian flow of \begin{equation}\label{eq:totalenergy}
H(q,\dot{q})= \frac12\langle \dot q, \dot q\rangle-U(q)
\end{equation} where the mass metric allows us to translate between $T^*\mathbb{C}^N$ and $T\mathbb{C}^N$ via $m_k\dot q_k = p_k$. In particular the \textit{total energy} $H$ is conserved over the motions. If we take the vector ${\textbf 1}_C:=(1,1,...,1)\in\mathbb{C}^N$ we have the additional conserved quantities:
\begin{equation}\label{eq:motion1}
L=\langle \dot q, {\textbf 1}_C\rangle_{\mathbb{C}} \qquad C=\omega(\dot q, q) = \langle \dot q, iq\rangle,
\end{equation} the \textit{linear momentum} and \textit{angular momentum}. Integrating the linear momentum once we may, as is standard, take the center of mass to be fixed at the origin
\[
\left\langle q, {\textbf 1}_C\right\rangle_{\mathbb{C}}= m_1q_1 + \dots  + m_Nq_N=0.
\]
Using the mass metric, the \textit{moment of inertia}, a measure of the total size of the configuration is given by
\[
I(q) = \left\langle  q, q \right\rangle= \|q \|^2.
\]

Now we consider \textit{central configurations}: those with $\nabla U(q)=\lambda q$.

Fix a configuration $q$. Make the Ansatz that there exists a homographic solution $z(t)q$ of eq. \eqref{eq:motion}. Then we have $$\ddot z |z|^{2+\alpha}z^{-1} q = \nabla U(q)$$ where $t$ only enters in the scalars on the left hand side. It follows that $$\nabla U(q)=\lambda q$$ for some constant\footnote{Note that by the rotation invariance of $U$, one must have $\lambda\in\mathbb{R}$, and by $U$'s homogeneity of degree $-\alpha$, one has $sgn(\lambda) =  sgn(-\alpha)$.} $\lambda$, and $z(t)$ is a solution of the $1/r^\alpha$ central force problem: $$\ddot z = \lambda z/|z|^{2+\alpha}.$$ That is to say, central configurations are exactly those which admit homographic solutions and the dynamics in the invariant plane $\mathbb{C}q$ are those of a $1/r^\alpha$ central force problem.

In particular, at a central configuration, we have the \textit{relative equilibria} solutions of eq. \eqref{eq:motion} rotating around the origin with angular velocity  $\omega\neq 0$, namely,
\[
q(t) =e^{i\omega t}q(0).
\]

Lastly, we recall the Jacobi-Maupertuis reformulation of mechanics (see \cite{Arnold}) which asserts that the solutions to Newton's equations at energy $H=h$ are, after a time reparameterization, precisely the geodesic equations for the
Jacobi-Maupertuis metric
\begin{equation}\label{JM}
ds^2_{JM}=(h+U)ds^2
\end{equation}
on the Hill region $\{q \in \mathbb{C}^N \setminus \triangle \,:\,h+U> 0\} \subset \mathbb{C}^N \setminus \triangle$, and with $ds^2$ the mass metric.

From now on, we will consider the configuration space $\mathbb{C}^N \setminus \triangle$ endowed with the Jacobi-Maupertuis metric and we denote by  $K_q(u,v)$  the sectional curvature of the JM-metric at $q$ through the plane spanned by $u, v$.

\section{Proofs of Results}
Before establishing Theorem \ref{inst} and Proposition \ref{prop3}, let us recall some background.

Given a Riemannian manifold, $(M^{2n},g)$ with metric compatible almost complex structure $J$, we split the complexified tangent space into the $i, -i$ eigenspaces of $J(v\otimes \lambda):=J(v)\otimes \lambda$, $$TM\otimes\mathbb{C}=TM'\oplus TM''.$$ In some local coordinates $(x^j,y^j)$ on $M$ s.t. $J(\partial_{x^j})=\partial_{y^j}$ we then have the bases for $TM'$ and $TM''$ respectively as $\partial_j:=\frac12 (\partial_{x^j}\otimes 1-\partial_{y^j}\otimes i)$ and $\overline \partial_j:=\frac12 (\partial_{x^j}\otimes 1+\partial_{y^j}\otimes i)$.

Now we extend the metric $\mathbb{C}$-linearly to a $\mathbb{C}$-valued symmetric bilinear form on $TM\otimes\mathbb{C}$, by $g(v\otimes \lambda, \cdot):=\lambda g(v,\cdot)$. Using the metric compatibility of $J$ we find $g_{ij}=g(\partial_i, \partial_j)=0=g_{\overline{ij}}$ and $$g_{i\overline j}=\frac12 (g(\partial_{x^i}, \partial_{x^j})+ig(\partial_{x^i},\partial_{y^j}))=g_{\overline i j}$$ and so $$g=g_{i\overline j}(dz^i\otimes d\overline z^j+d\overline z^j\otimes dz^i)=2g_{i\overline j} dz^id\overline z^j$$ where $dz^i=dx^i+idy^i, d\overline z^i=dx^i-idy^i$ are dual to $\partial_i, \overline\partial_i$.

The process above with the usual identification of $\mathbb{R}^2$
 with $\mathbb{C}$ by $z=x+iy$ and $J(x,y)=(-y,x)$ corresponding to multiplication by $i$, yields the familiar operators:\[
\frac{\partial}{\partial z}= \frac{1}{2}\left(\frac{\partial}{\partial x}-i\frac{\partial}{\partial y} \right),
\qquad \qquad
\frac{\partial}{\partial \overline z}= \frac{1}{2}\left(\frac{\partial}{\partial x}+i\frac{\partial}{\partial y} \right),
\] \[
dz=dx+idy,
\qquad \qquad
d\overline z=dx-idy,
\]

With this notation note that
\[
\Delta
= \frac{\partial^2}{\partial x^2}+\frac{\partial^2}{\partial y^2}
= 4 \frac{\partial^2}{\partial  z \partial \overline z},
\] \[
dzd\overline z= dx^2+dy^2.
\]
Let $\displaystyle{\partial=\frac{\partial}{\partial z}}$, $\displaystyle{\overline\partial=\frac{\partial}{\partial \overline z}}$. It then follows  that   a function $z\mapsto f(z)\in\mathbb{C}$ is holomorphic if and only if
$\overline \partial f(z)=0.$

Let $D\subset\mathbb{C}$ be a disk containing 0 and $f:D\to M$ a holomorphic map. Let $K_{f^*g}$ be the Gaussian curvature of $f^*g$ on $D$. For $X\in T_qM\otimes\mathbb{C}\setminus \{0\}$, define \begin{equation}\label{supdef}
    H_q(X):=\sup\{ K_{f^*g}(0) : f:D\to M\text{ holomorphic and } f(0)=q, \mathbb{C} f'(0)=\mathbb{C} X\}
\end{equation}
to be \textit{Kobayashi's holomorphic sectional curvature}. In (\cite{Kob}, ch. 2) Kobayashi  shows that
\begin{equation}\label{KHC}
    H_q(X) = R_{i\overline j k\overline l} X^i\overline X^j X^k\overline X^l
\end{equation}
for a unit vector $X=X^j\partial_j$ (that is $g(X,\overline X)=1$ by extending $g$ $\mathbb{C}$-linearly) and where
\begin{equation}\label{KRC}
    R_{i\overline{j} k\overline{l}} :=-\partial_k\overline \partial_l g_{i \overline{j}}+g^{p\overline{q}}\partial_k g_{i\overline{q}}\overline \partial_lg_{p\overline{j}}.
\end{equation}

To prove Proposition \ref{prop3} we will make use of the following Lemmas:

\begin{lem}
\label{holtokob}
Let $q\in\mathbb{C}^N\backslash\Delta$ and $v\in \mathbb{C}^N$. Complete $v$ to an orthonormal complex mass metric basis $v_1=v/\|v\|, v_2,...,v_N$. Then

$$K_q(v,iv) = H_q(v)-\frac{\sum_{j=2}^N |\partial_j U|^2}{(h+U(q))^3}$$

\end{lem}

\begin{proof}

Take an orthonormal Euclidean basis with $e_1=(1,0,...,0)=\lambda v_1$ and $e_2 = (0,1,0,...,0)=i\lambda v_1,...$ so that we may write the JM-metric as: $ds_{JM}^2 = (h+U)\sum \mu_j dz_jd\overline z_j$ where $\mu_j>0$ are some positive constants depending on the masses. In the real coordinates $(x_j, y_j)$ where $z_j=x_j+iy_j$, we have $ds_{JM}^2 = (h+U)\sum \mu_j (dx_j^2+dy_j^2)$ and $$K(v,iv)=\frac{R_{x_1y_1x_1y_1}}{|\partial_{x_1}\wedge\partial_{y_1}|_{JM}^2}=\frac{R_{x_1y_1x_1y_1}}{\mu_1^2 (h+U)^2}$$ where $R_{ijkl}$ is the Riemannian curvature tensor for $ds_{JM}^2$. In these real coordinates (see eq. \eqref{eq:curvature1}) we compute: $$R_{x_1y_1x_1y_1}=\frac{\mu_1}{2}(-(\partial_{x_1}^2 U + \partial_{y_1}^2U)+\frac{(\partial_{x_1} U)^2+(\partial_{y_1} U)^2}{h+U}-\frac{\sum_{j=2}^N \frac{\mu_1}{\mu_j}((\partial_{x_j} U)^2+(\partial_{y_j} U)^2)}{2 (h+U)})$$ and takes the form in complex coordinates ($\partial_j=\frac{\partial}{\partial z_j}=\frac12 (\partial_{x_j}- i \partial_{y_j})$) $$R_{x_1y_1x_1y_1}=\frac{\mu_1}{2} (-4\partial_1\overline\partial_1 U+ \frac{4 \partial_1 U\overline\partial_1 U}{h+U}-\frac{2\sum_{j=2}^N\frac{\mu_1}{\mu_j} \partial_j U\overline \partial_j U}{h+U}).$$


Now using eqs. \eqref{KHC} and  \eqref{KRC} with $v=\partial_1$ and corresponding unit vector $X=\sqrt{\frac{2}{\mu_1 (h+U)}} \partial_1=\frac{\partial_1}{\sqrt{g_{1\overline 1}}}$ we compute: $$R_{1\overline 1 1\overline 1}=\frac{\mu_1}{2}(-\partial_1\overline\partial_1 U +\frac{\partial_1 U\overline\partial_1 U}{h+U})$$ and then by eq. \eqref{KHC} $$H_q(v)=\frac{4R_{1\overline 1 1\overline 1}}{\mu_1^2 (h+U)^2}=K_q(v,iv)+\frac{\sum_{j=2}^N \frac{1}{\mu_j} \partial_j U\overline\partial_j U}{(h+U)^3}.$$ The formula in the Lemma then follows by rescaling $\partial_j$ by $\frac{1}{\sqrt{\mu_j}}$.

\end{proof}

\begin{lem}
\label{val} Let $q\in\mathbb{C}^N\backslash\Delta$, then the Kobayashi holomorphic sectional curvature (eq. \eqref{supdef}) associated to the JM-metric for a $1/r^\alpha$ potential ($\alpha> 0$) at energy level $h$ has $$H_q(q)=-\frac{h\alpha^2 U(q)}{2(h+U(q))^3\|q\|^2}.$$
\end{lem}

\begin{proof}
We will make use of the following general formula in the proof: Let $g=(g_1,...,g_k):D\to \mathbb{C}^k$ be holomorphic, then \begin{equation}\label{logCS} -\partial\overline\partial \log (c+ \| g\|^2)= \frac{|\langle g, g'\rangle |^2- \|g\|^2\|g'\|^2- c\| g'\|^2}{(c+\|g\|^2)^2}.
\end{equation} Let $f:D\to\mathbb{C}^N$ be a holomorphic map with $f(0)=q$ and $f'(0)=\lambda v$ (later we will set $v=q$). Then Kobayahi's holomorphic sectional curvature (eq. \eqref{supdef}) is the supremum over all such maps of the Gaussian curvature of $f^*(h+U)ds^2 = (h+U(f(z))\|f'(z)\|^2 dzd\overline z=:2\rho_f dzd\overline z$ at $z=0$. This Gaussian curvature is given by $-\frac{\partial\overline\partial \log\rho_f}{\rho_f}|_{z=0}$.

First we can see that the supremum is attained by a linear map $z\mapsto q+zv$ by considering \begin{equation}\label{lin}-\frac{\partial\overline\partial \log\rho_f}{\rho_f} = -\frac{\partial\overline\partial \log (h+U(f(z))}{\rho_f} - \frac{\partial\overline\partial \log \|f'(z)\|^2}{\rho_f}.\end{equation} Using that $f=(f_1,...,f_N)$ is holomorphic (so $\overline\partial f_k=0$) The first term of \eqref{lin} is $$\rho_f^{-1}(\frac{(\sum \overline\partial_j U \overline{\partial f_j})(\sum \partial_kU \partial f_k)}{(h+U)^2}-\frac{\sum\partial_k\overline\partial_j U \partial f_k\overline{\partial f_j}}{h+U})$$ which contains no second derivatives of $f$. As $f(0)=q$ is fixed and $f'(0)$ is fixed up to complex scaling, this first term when evaluated at $z=0$ is then independent of choice of $f$.

By applying eq. \eqref{logCS} to the second term of \eqref{lin}, we obtain $$\frac{|\langle f',f''\rangle|^2 - \|f'\|^2\|f''\|^2}{\|f'\|^4}$$ which by Cauchy-Schwarz achieves its supremum of 0 exactly when $f''(0)=\lambda f'(0)$, for instance when $f(z)=q+zv$.

To evaluate $H_q(\mathbb{C} v) = -2\frac{\partial\overline\partial \log(h+U(q+zv))}{(h+U(q))\|v\|^2}|_{z=0}$ set $$g_{ij}(z)=\frac{\sqrt{m_im_j}}{(q_i-q_j + z(v_i-v_j))^{\alpha/2}}$$ so that $U(q+zv) = \|g\|^2$ and $$g_{ij}'(z) = -\frac{\alpha}{2}\frac{\sqrt{m_im_j}(v_i-v_j)}{(q_i-q_j + z(v_i-v_j))^{1+\alpha/2}}.$$

Now taking $v=q$ we have $g'(0) = -\frac{\alpha}{2} g(0)$ and so by eq. \eqref{logCS} $$-\partial\overline\partial \log(h+ U(q+zq))|_{z=0} = -\frac{h\alpha^2 U(q)}{4(h+  U(q))^2}.$$  Hence $$H_q(q) = -\frac{h\alpha^2 U(q)}{2(h+U(q))^3\|q\|^2}.$$
\end{proof}

\begin{rem}\label{conj}

Let $q$ be a central configuration of a $1/r^\alpha$ potential and note that by Lemma \ref{holtokob}, and Cauchy-Schwarz inequality on the first two terms of eq. \eqref{logCS}, we have: $$K_q(v,iv)<-h\cdot Cst.$$  for $v\notin \mathbb{C}q$ and Cst. some positive constant (depending on $q$ and $v$). Now for $v\in \mathbb{C}q^\perp$ and from eq. \eqref{eq:curvature1}, we have $K_q(q,iq)+K_q(v,iv) = K_q(q,v) + K_q(iq,iv)$,  and so when $h\in [0,\infty)$: $$0>K_q(q,iq)+K_q(v,iv) = K_q(q,v) + K_q(iq,iv).$$ In particular, at least one of these sectional curvatures $K_q(q,v)$ or $K_q(iq,iv)$ is negative for $h\in [0,\infty)$. Because the curvatures depend continuously on the metric, whichever one is negative remains so for $h\in (-\epsilon^2, \infty)$.

In particular, for $0<\alpha<2$ and negative energies, we have a plane with negative sectional curvature tangent to the relative equilibria solution $\sigma = \text{span}\{ iq, iv\}$ or negative sectional curvature over the plane tangent to the total collapse solution: $\sigma = \text{span}\{ q, v\}$. Using that rotations are a symmetry of the metric, and scaling preserves signs of the sectional curvatures, the plane can be extended to a family of planes tangent to the solutions with negative sectional curvatures. 

\end{rem}

\begin{proof}[Proof of Proposition \ref{prop3}]

Take a complex orthonormal basis $v_1,...,v_N$ with $v_1=q/\|q\|$, then:

The configuration $q$ is a central configuration $\iff \nabla U(q) =\lambda q \iff 0 = \langle \nabla U (q), v_j\rangle = \partial_j U$ for $j>1 \iff H_q(q) = K_q(q,iq)$, where we have used Lemma \ref{holtokob} for the last equivalence. The value of $K_q(q,iq)$ along relative equilibria is then given by the computation in Lemma \ref{val}.

\end{proof}

\begin{proof}[Proof of Theorem \ref{inst}]

The proof will use the following properties of geodesic flow and Jacobi fields.  First, let $\gamma$ be a unit speed geodesic in a surface and $J = \lambda \dot\gamma^{\perp}$ a normal Jacobi field along $\gamma$. Then $\ddot\lambda = -K_{\gamma}(\dot\gamma, J)\lambda$.  Second, the linearization of the geodesic flow, $\phi_t:TM\to TM$, is given by $$d\phi_t(\xi) = (J_\xi(t), \dot J_{\xi}(t)),$$ where $J_\xi(t)$ is the unique Jacobi field whose initial condition $(J_\xi(0), \dot J_\xi(0))$ corresponds to $\xi\in TTM$.  In particular, to examine the Poincar\'e return map of $\phi_t$ along a closed geodesic $\gamma$ we will examine the growths of Jacobi fields. Now let $q$ be a central configuration for a $1/r^{\alpha}$ potential with $\alpha\ge 2$ and consider the totally geodesic surface $\mathbb{C}^{*}q$.

First we take the case when $\alpha > 2$. The only periodic homographic motions are the relative equilibria, occurring at positive energy levels. Let $\gamma(t) = e^{i \omega t}q$ be the relative equilibrium through $q$ at energy level $h>0$. With $\omega$ chosen so that $\gamma$ is a unit speed geodesic of the JM-metric.

Take a normal Jacobi field $J(t) = \lambda(t) \gamma(t)$. By Proposition \ref{prop3}, and the rotation symmetry of the metric, $$K_{\gamma}(\dot\gamma, J) = K_{e^{i\omega t}q}(ie^{i \omega t}q, e^{i \omega t}q) = K_{q}(iq, q) =-c^2<0.$$ Then $\lambda$ satisfies the second order differential equation $\ddot\lambda = c^2\lambda$.  For the normal Jacobi field with initial condition $J(0) = q, \dot J(0) = cq$, we have $J(2\pi / \omega) = e^{2\pi c/\omega} J(0), \dot J(2\pi /\omega) = e^{2\pi c/\omega}\dot J(0)$ and thus an eigenvalue $e^{2\pi c/\omega} \ne 1$ of the return map. These relative equilibria are spectrally instable.

Now we take $\alpha = 2$. Again the only periodic homographic motions are relative equilibria, but now they occur when $h=0$. Taking the same notations as in the $\alpha >2$ case, we have $\ddot \lambda = 0$. Hence a normal Jacobi field with initial condition $J(0), \dot J(0)$ has $J(2\pi /\omega) = 2\pi /\omega \dot J(0)+ J(0), \dot J(2\pi /\omega) = J(0)$. Hence we have a non-diagonalizable Jordan-block of the return map. These relative equilibria are linearly instable.

\end{proof}

In order to establish the remaining results, it will be necessary to use some auxiliary results. The first one is about the sectional curvature of the configuration space and can be found in \cite{CM}.

\begin{lem} \label{lem:JM}
The sectional curvature of $\mathbb{C}^{N}\setminus \triangle$ endowed with the  Jacobi-Maupertuis metric $(h+U)ds^2$ is given by
\begin{equation}
\label{eq:curvature1}
(h+U)^3 K(\sigma)
=
\frac{3}{4}\left((\partial_1 U)^2+(\partial_2 U)^2\right)- \frac{1}{4} \left\|\nabla U\right\|^2-\frac{h+U}{2}\left(\partial_1^2 U + \partial_2^2 U\right),
\end{equation}
where $\partial_a U$ denotes $dU(v_a)$ and $a=1,2$ with
$v_1, v_2$ are $ds^2$-orthonormal vectors spanning $\sigma \subset \mathbb{C}^{N}$. The $\left\|\,\right\|$ and
$\nabla$ refer to the norm and Levi-Civita connection for the mass metric.
\end{lem}

We will see that the expression \eqref{eq:curvature1} becomes manageable when the plane $ \sigma $ is spanned by the vectors $ \dot {q} (t) $, $ \textbf {1}$. The constant vector $\textbf{1}$ is  formed by the complex numbers of the form $1 + i$ in each coordinate. More specifically,  we  will need the following lemma for the sectional curvature on $\sigma$. Let us observe that Proposition \ref{prop1}  is a direct consequence of the following result.

\begin{lem}\label{lem:aux1}
Let $q(t)$ be a solution of the $N$-body problem with an attractive $1/r^\alpha$ potential with center of mass zero. Suppose that $\dot{q}(t) \neq 0$ for all $t$ where the curve $q(t)$ is defined.
Then the sectional curvature $K_{q(t)}(\dot q(t), \textbf{1})$ of the Jacobi-Maupertuis metric along $q(t)$, satisfies:
\begin{equation}
\label{eq:curvature2}
8(h+ U)^4 K_{q(t)}( \dot{q}(t), \textbf{1})
=3\left( \frac{dU}{dt} \right)^2 -2(h+U)\frac{d^2U}{dt^2}.
\end{equation}
\end{lem}
\begin{proof}
Consider the directions $v_1=\frac{\dot{q}}{\|\dot{q}\|}, v_2=\frac{\textbf{1}}{\|\textbf{1}\|}$. Since the center of mass as fixed at the origin of the system, the vectors $v_1$ and $v_2$ are $ds^2$-orthonormal.
 It follows directly that
\begin{equation}\label{eq:1lem}
\partial_1 U=\left\langle \nabla U,v_1 \right\rangle=\frac{\left\langle\nabla U,\dot{q} \right\rangle}{\|\dot{q}\|}=\frac{\left\langle\ddot{q},\dot{q} \right\rangle}{\|\dot{q}\|}.
\end{equation}

We consider the functions $g_k(t):=\nabla_{x_k} U(q(t))=\ddot{x}_k(t)$, for $k=1,\ldots,2N,$ where $q=(x_1,\ldots,x_{2N})$.
Since $\nabla_{x_k} U = \frac{1}{m_k}  \frac{\partial U}{\partial x_k}$ we have
\[
\dot{g}_k(t)= \frac{1}{m_k}\sum_{l=1}^{2N}\dot{x}_l(t) \frac{\partial^2 U}{\partial x_k \partial x_l}(q(t))=\overset {...}{x}_k(t)
\]
and
\[
\left\langle\dot{q},\overset {...}{q} \right\rangle=\left\langle \dot{q},(\dot{g}_1,\dots,\dot{g}_{2N})\right\rangle=\sum_{k,l=1}^{2N} \dot{x}_k(t) \dot{x}_l(t)
\frac{\partial^2 U}{\partial x_k \partial x_l}(q(t)).
\]
On the other hand, we also have that
\[
\left\langle \nabla_{\dot{q}}\left\langle \nabla U,\dot{q}\right\rangle,\dot{q} \right\rangle = \sum_{k,l=1}^{2N} \dot{x}_k(t) \dot{x}_l(t)
\frac{\partial^2 U}{\partial x_k \partial x_l}(q(t)).
\]
Thus, we obtain the relation $
\left\langle \nabla_{\dot{q}}\left\langle \nabla U,\dot{q}\right\rangle,\dot{q} \right\rangle= \left\langle\dot{q},\overset {...}{q} \right\rangle.$
Now we can determine the term $\partial_1^2 U$:
\[
\partial_1^2 U= \left\langle\nabla_{\dot{q}} \left\langle \nabla U,v_1\right\rangle,v_1 \right\rangle
=\|\dot{q}\|^{-2}\left\langle \nabla_{\dot{q}}\left\langle \nabla U,\dot{q}\right\rangle,\dot{q} \right\rangle.
\]
Thus,
\begin{equation}\label{eq:2lem}
\partial_1^2 U=\frac{\left\langle \dot{q},\overset {...}{q} \right\rangle}{\|\dot{q}\|^2}.
\end{equation}
Let us now consider the term $\partial_2 U$. Since $\nabla U \in \{q\in \mathbb{C}^N \setminus \triangle \,:\,\sum m_k q_k =0\}$, we obtain
\begin{equation}\label{eq:3lem}
\partial_2 U=\left\langle \nabla U,v_2 \right\rangle =\frac{\left\langle \nabla U,\textbf{1} \right\rangle}{\|\textbf{1}\|}=0.
\end{equation}
Similarly,  $\partial_2^2 U=0$.

Substituting  \eqref{eq:1lem}, \eqref{eq:2lem}, \eqref{eq:3lem} into \eqref{eq:curvature1} leads to the equation
\[
4 (h+U)^3  K_{q(t)}( \dot{q}(t), \textbf{1})=\frac{3}{4} \frac{ \langle \ddot{q},\dot{q} \rangle^2}{\| \dot{q}\|^2}-\frac{1}{4} \|\ddot{q}\|^2 - \frac{h+U}{2}  \frac{ \langle \overset {...}{q},\dot{q} \rangle}{\| \dot{q}\|^2}.
\]
On the other hand, from \eqref{eq:totalenergy}, we get $ 2(h+U)= \|\dot{q}\|^2$. Therefore,  multiplying the above equation by $4\|\dot{q}\|^2$ we obtain
\begin{equation}\label{eq:perro}
8 (h+U)^4  K_{q(t)}( \dot{q}(t), \textbf{1})=3 \langle \ddot{q},\dot{q} \rangle^2 -2 (h+U ) ( \langle \ddot{q},\ddot{q} \rangle +   \langle \overset {...}{q},\dot{q} \rangle).
\end{equation}
Since the total energy  is constant along solutions, $ \frac{dH}{dt}=0= \langle \ddot{q},\dot{q} \rangle-\frac{dU}{dt} $.
Thus, substituting
\[
\frac{dU}{dt}  = \langle \ddot{q},\dot{q} \rangle \quad\text{and}\quad
 \frac{d^2U}{dt^2 }  =\| \ddot{q} \|^2 + \langle \overset {...}{q},\dot{q} \rangle
\
\]
  into \eqref{eq:perro} yields \eqref{eq:curvature2}.
\end{proof}

\begin{rem}
It is not necessary take  $v_2$ as $\frac{\bf{1}}{\| \bf{1} \|}$, but only that the unit vector $v_2$ satisfies $\partial_2 U=\partial_2^2 U=0$
along the $q(t)$. Note that $\partial_2 U= \sum_{k} \eta_k \frac{\partial U}{\partial x_k}$ where $v_2=(\eta_1, \cdots, \eta_{2N})$. 
\end{rem}

\begin{lem}\label{lem:gato}
Under the hypothesis of Lemma \ref{lem:aux1}, the equation
$
K_{q(t)}( \dot{q}(t), \textbf{1})=0$, with $\frac{dU}{dt} \neq 0,$
is equivalent to the first integral
\begin{equation}\label{eq:first-integral}
C \left(  \frac{dU}{dt }  \right)^2 =(h+U)^3,
\end{equation}
where $C$ is a positive integration constant.
\end{lem}
\begin{proof}

Suppose $K_{q(t)}(\dot q(t), \textbf{1}) =0$, by \eqref{eq:curvature2}, this is equivalent to $\frac{3}{h+U} \left(\frac{dU}{dt}\right)^2=2 \frac{d^2U}{dt^2}$. Now multiply both sides by $ \left(\frac{dU}{dt}\right)^{-1}$ and integrating yields $3\log (h + U) = 2\log \frac{dU}{dt} + \log C$, i.e. \eqref{eq:first-integral}. 
\end{proof}

We will use the Lagrange-Jacobi identity for the homogeneous potential $-U$ of degree $-\alpha$:
\begin{equation}\label{eq:LJ}
\ddot{I}=4H + (4-2\alpha)U
\end{equation}

\begin{proof}[Proof of Proposition \ref{prop2}]
Let $q$ be a bounded solution of the planar $N$-body problem with an attractive $1/r^\alpha$ potential, with $\alpha=2$. The boundedness of $q$ and the Lagrange-Jacobi identity implies the total energy is zero.
We suppose 
 that $\displaystyle{\frac{dU}{dt}\neq 0}$ for some $t$. By Lemma \ref{lem:gato}, we know that $K_{q(t)}(\dot{q}(t),\textbf{1})=0$ is equivalent to the first integral
\[
\displaystyle{C\left(\frac{dU}{dt}\right)^2=U^3},
\]
  where $C$ is some positive constant.
  To solve this separable equation, we rewrite it in the form
  \[
  \sqrt{C} \frac{dU}{U^{3/2}} = dt.
  \]
  Integrating both sides gives $U(t)=\frac{4C}{(t+A)^2}$
  where $A$ is some constant. In particular, we have $\{U(t)) \}=\mathbb{R}^+$. Hence $U(t)\rightarrow 0$  
  implies:
\[
\limsup r_{ij}= \infty, \qquad \text{for all }\, 1\leq i<j\leq N,
\]
giving a contradiction. Indeed, the moment of inertia can be written as
\[
\displaystyle{I=\frac{1}{M}\sum_{i<j}  m_i m_j r_{ij}^2},
\] where $M=m_1+\cdots+m_N$,
which by hypothesis is constant.
It follows that $U(t)$ must be constant.
\end{proof}


For the proof of Corollary \ref{coro1}, in addition to Proposition \ref{prop2} we also need the following result given by Fujiwara et. al    \cite[Theorem 1]{Fuji15}, which is the solution to the  Saari's homographic conjecture in the planar $3$-body
problem for general masses with $\alpha=1$ and $\alpha=2$. In our case we only need this result for $\alpha = 2$.

\begin{teo}\label{fuji}\cite{Fuji15}
For the planar $3$-body problem with  $\alpha= 2$.  If a motion has constant $\mu=I U$, then the motion is homographic.
\end{teo}

\begin{proof}[of Corollary \ref{coro1}]
Let  $\alpha=2$. Let $q(t)$ be a  solution of the planar $3$-body problem with  $I(t)= $constant and $\displaystyle{K_{q}(\dot{q}, \textbf 1)}=0$.
Again, from the Lagrange-Jacobi identity,  we have that  $H=h = 0$. Hence, by Proposition \ref{prop2} the solution has  constant potential energy. By  Theorem \ref{fuji}, it is immediate that   the motion $q(t)$ must be  a relative equilibrium.
\end{proof}

\section{Questions}

When $\alpha=2$ and $h=0$, does the non-positive holomorphic sectional curvature persist on the reduced space, that is under the quotient by translations and complex scaling, and would there be any dynamical consequences of such negative holomorphic sectional curvature?

We recall that for $\alpha =2$, the reduced space is the quotient of the conguration space $\mathbb{C}^N\setminus\Delta$  by the translations, rotations \textit{and scalings}, see e.g. \cite[Section 2]{CM}. Only when $\alpha=2$ and $h=0$, is the scaling also a symmetry of the JM-metric. For $N>3$ and sectional curvatures through arbitrary planes, the answer to the question is no, i.e. there are $2$-planes in the reduced space at which the sectional curvature is positive, see again \cite[Theorem 1]{CM}. However, in \cite{JM} we
proved that the parallelogram subproblem, which corresponds to a totally geodesic
two-dimensional surface within the reduced space, has non-positive Gaussian curvature. The tangent planes to this surface of parallelogram configurations are complex planes. See also \cite{M2005, M} for more on the reduction process
and some dynamical consequences for planar three-body problems with different potentials.

\section*{Acknowledgments}
Thanks to Wei Yuan of Sun Yat-Sen University, who first brought attention to the holomorphic sectional curvature and Alain Chenciner for lending his copy of Kobayashi's Hyperbolic Complex Spaces. Also we thank Richard Montgomery and Jacques F\'ejoz for their encouragement to pursue dynamical consequences and Rick Moeckel and Manuele Sorprete for their interest and comments. Josu\'e Mel\'endez  is partially supported by SEP-PRODEP, UAM-PTC-638, M\'exico. Connor Jackman was supported by the National Science Foundation under Grant No. DMS-1440140, and the National Security Agency under Grant No. H98230-18-1-0188. We are grateful to the referees for a careful reading of theorem 1 which led us to realize some serious gaps in the original proof.


\begin{thebibliography}{99}



\bibitem{problems} A. Albouy, H. Cabral, A. Santos, {\em Some Open problems on the classical N-body problem} Celest Mech Dyn Astr (2012) 113: 369. https://doi.org/10.1007/s10569-012-9431-1

\bibitem{Arnold} V.I. Arnold, {\em Mathematical methods of classical mechanics}, 2nd ed., New York: Springer 1989.

\bibitem{Barbosu} M. Barbosu, B. Elmabsout, {\em Courbures de Riemann dans le probl\`eme plan des trois corps - Stabilit\'e}, C.R. Acad. Sci. Paris, t. 327, S\'erie II b, p. 959-962, 1999.

\bibitem{Bar} Barutello, Vivina, Riccardo D. Jadanza, and Alessandro Portaluri.``Morse index and linear stability of the Lagrangian circular orbit in a three-body-type problem via index theory." Archive for Rational Mechanics and Analysis 219.1 (2016): 387-444.

\bibitem{Chenc} A. Chenciner, {\em Collisions totales, mouvements compl\`etement paraboliques et r\'eduction des homoth\'eties dans le probl\`eme des n corps}, Regular and chaotic dynamics, V.3, 3, pp. 93-106 (1998).

%
%
%



\bibitem{Fuji15}  T. Fujiwara, H. Fukuda, H. Ozaki, T. Taniguchi,  {\em Saari's homographic conjecture for general
masses in planar three-body problem under Newton potential and a strong force potential},
J. Phys. A: Math. Theor. {\bf 48} (2015) 265--205.

\bibitem{Hu} Hu, Xijun, and Shanzhong Sun. ``Morse index and stability of elliptic Lagrangian solutions in the planar three-body problem." Advances in Mathematics 223.1 (2010): 98-119.

\bibitem{Aniso} Hu, Xijun, and Guowei Yu. ``Index Theory for Zero Energy Solutions of the Planar Anisotropic Kepler Problem." arXiv preprint arXiv:1705.05645 (2017).


\bibitem{JM} C. Jackman and J. Mel\'endez.
{\em Hyperbolic Shirts fit a 4-body problem}, Journal of Geometry and Physics. 123 (2018) 173--183.

\bibitem{CM} C. Jackman and R. Montgomery.
{\em No hyperbolic pants for the 4-body problem with strong potential}, Pacific J. Math. {\bf 280} (2016), 401--410.

\bibitem{Kob} S. Kobayashi. {\em Hyperbolic Complex Spaces}, vol 318 of Grundlehren der matimatischen Wissenschaften, Berlin: Springer-Verlag1998.

\bibitem{Lee} Lee, J. M.,  {\em Riemannian Manifolds: An Introduction to Curvature}, Springer-Verlag Inc, (1997).


\bibitem{MS} Mart\'inez, R.,  Sama, A., and Sim\'o, C. ``Stability of homographic solutions of the planar three-body problem with homogeneous potentials." Eq. Diff. 2003. 2005. 1005-1010.


\bibitem{MoeckelCC} R. Moeckel. {\em Lectures on central configurations} \newline
 http://www.math.umn.edu/~rmoeckel/notes/CentralConfigurations
 
 \bibitem{MoS} R. Moeckel, {\em A computer assisted proof of Saari's conjecture for the three-body problem in} $\mathbb{R}^d$. Discrete \& Continuous Dynamical Systems - S, 2008, 1 (4) : 631-646.
 
 \bibitem{M} R. Montgomery. {\em The hyperbolic plane, three-body problems, and Mn\"ev's universality theorem.} Regul. Chaotic Dyn. {\bf 22} (2017),  6, 688--699.

\bibitem{M2005} R. Montgomery. {\em Fitting hyperbolic pants to a three-body problem}, Ergodic Theory Dynam. Systems {\bf 25} (2005), 921--947.

\bibitem{Mblowup} R. Montgomery. {\em Blow-up for realizing homotopy classes in the three-body problem}, (expository) for a school/conference held at CIMAT in Feb 2015, see arXiv:1507.07982.

\bibitem{Ong} Pin, Ong Chong. {\em Curvature and mechanics.} Advances in Mathematics 15.3 (1975): 269-311.

\bibitem{Paternain} Paternain, Gabriel P. Geodesic flows. Vol. 180. Springer Science \& Business Media, 2012.

\bibitem{GR} G. Roberts, {\em Some Counterexamples to Saari's Generalized Conjecture}, Transactions of the American Mathematical Society, Vol. 358 no. 1 (2005) 251 --265

\bibitem{Roberts} Roberts, Gareth E. ``Linear stability of the elliptic Lagrangian triangle solutions in the three-body problem." Journal of Differential Equations 182.1 (2002): 191-218.



%



\end{thebibliography}
\end{document}